\documentclass[11pt,a4paper]{amsart}
\usepackage{amsmath,amsfonts,amsthm,amsopn,color,amssymb,enumitem}
\usepackage{palatino}
\usepackage{graphicx}
\usepackage[english]{babel}
\usepackage{caption}
\usepackage{subcaption}
\usepackage[colorlinks=true]{hyperref}
\hypersetup{urlcolor=blue, citecolor=red, linkcolor=blue}
\usepackage[utf8]{inputenc}
\usepackage{colortbl}

\usepackage{cite}
\usepackage{esint}

\numberwithin{equation}{section}

\newtheorem{theorem}{Theorem}[section]
\newtheorem{lemma}[theorem]{Lemma}

\newtheorem{definition}[theorem]{Definition}
\newtheorem{proposition}[theorem]{Proposition}
\newtheorem{remark}[theorem]{Remark}

\newtheorem{thmx}{Theorem}

\newcommand{\e}{\varepsilon}

\renewcommand{\o}{\omega}

\newcommand{\R}{\mathbb{R}}

\newcommand{\Sp}{\mathbb{S}^{N-1}}
\newcommand{\Sr}{\mathbb{S}^{N-1}(r)}

\newcommand{\vt}{\vartheta}

\renewcommand{\l }{\lambda}

\newcommand{\N}{\mathbb{N}}

\renewcommand{\o}{\omega}

\newcommand{\beq }{\begin{equation}}
\newcommand{\eeq }{\end{equation}}

\begin{document}

\title[Nonsymmetric solutions to overdetermined problems in bounded domains]{Nonsymmetric sign-changing solutions to overdetermined elliptic problems in bounded domains}

\author{David Ruiz}
  \address{David Ruiz \\
    IMAG, Universidad de Granada\\
    Departamento de An\'alisis Matem\'atico\\
    Campus Fuentenueva\\
    18071 Granada, Spain}
  \email{daruiz@ugr.es}

\thanks{D. R. has been supported by:
the Grant PID2021-122122NB-I00 funded by MCIN/AEI/ 10.13039/501100011033 and by
``ERDF A way of making Europe'';
	the Research Group FQM-116 funded by J. Andalucia;
the \emph{IMAG-Maria de Maeztu} Excellence Grant CEX2020-001105-M/AEI/10.13039/501100011033 funded by MICIN/AEI.}




\begin{abstract}
In 1971 J. Serrin proved that, given a smooth bounded domain $\Omega \subset \R^N$ and $u$ a positive solution of the problem:

\begin{equation*}
	\begin{cases}
		-\Delta u = f(u) &\mbox{in $\Omega$, }\\
		u =0 &\mbox{on $\partial\Omega$, }\\
	\partial_{\nu} u =\mbox{constant} &\mbox{on $\partial\Omega$, }
	\end{cases}
\end{equation*}
then $\Omega$ is necessarily a ball and $u$ is radially symmetric. In this paper we prove that the positivity of $u$ is necessary in that symmetry result. In fact we find a sign-changing solution to that problem for a $C^2$ function $f(u)$ in a bounded domain $\Omega$ different from a ball. The proof uses a local bifurcation argument, based on the study of the associated linearized operator. 

\end{abstract}

\maketitle


\section{Introduction}

This paper is concerned with solutions to semilinear elliptic problems in the form:
\begin{equation}\label{first}
	\left\{\begin{array} {ll}
		-\Delta u = f(u) & \mbox{in }\; \Omega,\\
			u  = 0 & \mbox{on }\; \partial \Omega, \\
		\partial_{\nu} u =\mbox{constant} &\mbox{on }\; \partial
		\Omega.
	\end{array}\right.
\end{equation}
Here $\Omega$ is a regular domain in $\R^N$, $f$ is a Lipschitz continuous function and $\partial_{\nu} u$ denotes the normal derivative of $u$. These problems are called overdetermined because of the two boundary conditions, and appear quite naturally in the study of free boundaries in many different phenomena in Physics, like in capillarity, elasticity and others (see \cite{S02,S56} for more details).

Because of the two boundary conditions one does not expect, in general, to obtain generic existence results for arbitrary domains. Indeed, many rigidity results are available in the literature. The first one is the result of J. Serrin in 1971 (\cite{serrin}), who proved that if $\Omega$ is bounded and $u$ is positive, then necessarily $\Omega$ is a ball and $u$ is radially symmetric. Serrin's proof is based on the moving plane method, introduced in 1956 by A. D. Alexandrov in \cite{A56} to prove that the only compact embedded hypersurfaces in $\R^N$ with constant mean curvature are the spheres. This is also the first example of a recurrent analogy between overdetermined elliptic problems and constant mean curvature (CMC) surfaces.  This result has been later extended to fully nonlinear operators (\cite{mira, si-si}).  
\medskip

The case of unbounded domains $\Omega$ has also attracted much attention, and some rigidity results have been given, see \cite{BCN97, FV10, R97, RRS17, sir, WW}. Under certain assumptions, in those papers the authors show that positive solutions to \eqref{first} depend only on one variable, or they are radially symmetric, depending on the case under study. However, for unbounded domains $\Omega$ the situation turns out to be less rigid and nontrivial solutions exist. A first such example was found by Sicbaldi in \cite{S10}, where it is shown that suitable periodic perturbations of the straight cylinder $B^{N-1}\times\mathbb{R}$ support periodic solutions to the problem (\ref{first}) with $f(u)=\lambda u, \lambda>0$. This solution can be seen as an analogue of the onduloid and strengthen the analogy between overdetermined problems and CMC surfaces. After such first construction, other examples of nontrivial solutions have been obtained, see for instance \cite{DPW15, FMW17, RRS20, RSW1, SS12}. 

Overdetermined elliptic problems posed in complete Riemannian manifolds have also been studied, see \cite{DS15, DEP19, FMV13, PS09, S14}. In this framework, the case of subdomains $\Omega$ of the sphere $\mathbb{S}^N$ or the hyperbolic space $\mathbb{H}^N$ is of special interest. A complete counterpart of the Serrin's theorem has been obtained in \cite{KP98} if $\Omega$ is a bounded domain in $\mathbb{H}^N$ or in the semisphere $\mathbb{S}^{N}_+$. Another rigidity result is given in \cite{EM19} for simply connected domains $\Omega \subset \mathbb{S}^2$ and for certain nonlinearities $f(u)$. On the other hand, nontrivial examples in the sphere have been constructed (see\cite{FMW18, RSW2}). 

\medskip

All the previous discussion is concerned with positive solutions of \eqref{first}. Sign-changing solutions, however, are also important in the applications and have been the object of intensive research. Maybe the most most challenging open question in the framework of overdetermined problems is the Schiffer conjecture, which can be written as follows: 

\medskip

{\bf Schiffer Conjecture}: Let $\Omega \subset \R^N$ be a bounded regular domain, and $u: \Omega \to \R$ a nonconstant solution to the problem:
$$ \left\{ \begin{array}{lr}  \Delta u + \lambda u  =0 & \mbox{in } \Omega,\\ u= c &  \mbox{ on } \partial \Omega, \\ \frac{\partial u}{\partial \nu}=0 &  \mbox{ on } \partial \Omega. \end{array} \right. $$
Then $\Omega$ is a ball and $u$ is radially symmetric. \

\medskip

If we define $\tilde{u}=u - c$ we are led with a problem in the form \eqref{first}, without any sign restriction on the function $\tilde{u}$. The Schiffer conjecture has attracted much attention since it turns out to be equivalent to the resolution of the so-called Pompeiu problem, see \cite{williams, Zalcman}.

Another motivation to study problems like \eqref{first} comes from stationary solutions of Euler equations, see for instance \cite{enciso, yo, hamel4}. We emphasize that, in this framework, the function $u$ need not be positive.

%
%
%
%
%
%
%

When considering sign-changing solutions, maybe the most natural question is whether the original result of Serrin is true also without the positivity assumption. This problem has remained completely open so far. The main goal of this paper is to give a negative answer to that question. Indeed we will prove the following result: 

\begin{theorem}\label{0main} Let $N=2$, $3$ or $4$. There exist bounded smooth domains $\Omega \subset \R^N$ different from a ball such that the problem:
\begin{equation} \label{0}
	\begin{cases}
		-\Delta u = u-(u^+)^3  &\mbox{in $\Omega$, }\\
		u=0 &\mbox{on $\partial\Omega$, }\\
		\partial_{\nu} u=\mbox{constant}\neq 0 &\mbox{on $\partial\Omega$, }
	\end{cases}
\end{equation}
admits a sign-changing solution.
\end{theorem}

The proof of Theorem \ref{0main} makes use of a local bifurcation argument. For that, we first build a family of sign-changing radial solutions to the problem: 
\begin{equation*}
	\begin{cases}
		- \rho \, \Delta u = u - (u^+)^3 & \mbox{ in }B, \\
		u=0 &\mbox{on }\partial B. \end{cases}
\end{equation*}
Here $B$ is the unit ball and $\rho \in (0, \bar{\rho})$, where $\bar{\rho}$ is a fixed positive constant. In our argument we need that the family of solutions $u_\rho$ is smooth with respect to $\rho$, which is a consequence of the nondegeneracy of the solutions with respect to radial variations of the Dirichlet problem. We also need nondegeneracy of the Dirichlet problem even for nonradial variations, at least for $\rho$ in a certain subinterval $I\subset (0, \bar{\rho})$. This allows us to build a nonlinear Dirichlet-to-Neumann operator $F_\rho$ whose zeroes correspond to solutions of \eqref{0} (see \eqref{eqF}). Then, we are able to prove that the corresponding linearized operator $DF_\rho$ becomes degenerate at a certain value of $\rho \in I$, and that its multiplicity is odd. At this point we are under the conditions of the Krasnoselskii bifurcation theorem to prove local bifurcation. A rescaling of the dependent variable allows us to conclude the proof of Theorem \ref{0main}.

The main difficulty in the proof is to devise an appropriate choice of nonlinear term $f(u)$, and a convenient family of solutions $u_\rho$, with all the properties described above. Regarding the nonlinearity, the fact that $f(u)=u-(u^+)^3$  is linear for $u<0$, but of Allen-Cahn type for $u>0$, is essential in our arguments. Minor modifications in the proofs would give the same result for $f(u)= u - (u^+)^p$, $p>1$. Regarding the family of solutions, we also need to know the asymptotic behavior of $u_\rho$ as $\rho \to 0$: this is useful for proving the crossing of eigenvalues of the linearized operator.

In order to study the nondegeneracy of the Dirichlet problem and the eigenvalues of $DF_\rho$, we need to deal with linear operators defined in $B$ with radially symmetric coefficients. By separation of variables, the eigenfunctions come as radial functions multiplied by spherical harmonics. In order to show that the first eigenvalue of $DF_\rho$ is positive for some $\rho$ we need to exclude some of the first spherical harmonics of the sphere. We do that by imposing a (suitably large) symmetry group $G$ in our functional setting (see condition \ref{G} for details). In this way, we also rule out the degeneracy of the problem due to its invariance by translations in $\R^N$. 

For $N=2$, it suffices to consider a dihedral symmetry group $\mathbb{D}_k$, $k \geq 5$. If $N=3$, the group of all isometries of the icosahedron satisfies our assumption \ref{G}. We have also found an appropriate symmetry group in dimension 4, the group of rotations of the hyper-icosahedron. The hyper-icosahedron is a regular polytope in $\R^4$ with 600 tetrahedral cells and 120 vertices, and its group of rotations forms a 7200 elements subgroup of $SO(4)$ (see \cite{widom}). It is worth pointing out that the set of regular polytopes in dimension $N \geq 5$ is limited to three (the hyper-tetrahedron, the hyper-cube and the hyper-octahedron), and their symmetry groups do not seem to satisfy \ref{G}. It is not clear to us if there exists a different type of symmetry group (not related to regular polytopes) satisfying condition \ref{G} also for $N \geq 5$; in such case, our result extends immediately to dimension $N$.

We finish this introduction by commenting some consequences of Theorem \ref{0main}. As was commented above, overdetermined elliptic problems appear also in the study of fluid equations, in particular of stationary solutions of the 2D Euler equations. In this framework, the following result has been recently obtained by Hamel and Nadirashvili (\cite[Theorem 1.10]{hamel4}).

\begin{thmx}[\cite{hamel4}] \label{teo} Let $v: \overline{\Omega} \to \R^2$ be a $C^2$ solution of the stationary Euler equations:
	
\begin{equation} \label{2D} \left \{ \begin{array}{ll}  v \cdot \nabla v & = - \nabla p, \\ div \, v & =0.\end{array}  \right. \end{equation}

Here $\Omega$ is a bounded regular and simply connected domain in $ \R^2$. Assume also that:
	
	\begin{enumerate}
		\item $v \cdot \nu =0$ and $|v|$ is a nonzero constant on $\partial \Omega$.
		\item There exists a unique point $p \in \Omega$ such that $v(p)=0$.
	\end{enumerate}
	
	Then $\Omega$ is a ball centered at $p$ and $v$ is a circular vector field.
	
\end{thmx}

As a consequence of Theorem \ref{0main} we conclude that assumption (2) of Theorem \ref{teo} is necessary. Indeed, if $u$ is a solution as given by Theorem \ref{0main}, then $v= \nabla^{\perp} u $ solves \eqref{2D} and satisfies (1). Our solutions have several critical points, though, which correspond to stagnation points of the fluid. In the case of pentagonal symmetry, for instance, it is to be expected that our solution $u$ has at least 11 critical points (1 maximum, 5 minima and 5 saddle points). 

Moreover, if $u$ is as in Theorem \ref{0main}, the function:
$$ v(x)= \left \{ \begin{array}{ll} \nabla^{\perp} u(x) &  x \in \Omega \\ 0 & x \notin \Omega  \end{array} \right.$$
is an example of a compactly supported (discontinuous) weak solution of the Euler equations \eqref{2D} in the whole plane $\R^2$. See \cite{gs2} for a continuous (piecewise $C^1$) solution without radial symmetry.

The rest of the paper is organized as follows. In Section 2 we establish the notation and some preliminary results. We also give a more precise statement of our main result in Theorem \ref{main}. Section 3 is devoted to prove the existence of the family of radial solutions $u_\rho$ and to study their properties. We also study the behavior of the linearized operator with Dirichlet boundary conditions. The definition of the Dirichlet-to-Neumann operator $F_\rho$ and the computation of its linearization $DF_\rho$ is performed in Section 4. In Section 5 we show that the first eigenvalue of $DF_\rho$ crosses the $0$ value as $\rho$ varies. This allows us to finish the local bifurcation argument in Section 6. Finally, in the Appendix the assumption made on the symmetry group $G$ is studied.

\medskip 

{\bf Acknowledgment: } The author wishes to express his gratitude to Pieralberto Sicbaldi and Daniel Peralta Salas for bringing this problem to his attention, and also for many discussions on this and related topics. He also thanks Michael Widom for his gentle help in intepreting his paper \cite{widom}.

%
%
%
%
%
%
%
%
%

\section{Notations and preliminaries}

We denote by $B(R)\subset \R^N$ the ball centered at the origin of radius $R$, and $B=B(1)$. Analogously we define the unit sphere centered at the origin $\Sp$ and the sphere of radius $r$, $\Sp(r)$. We write $G$ to denote a symmetry group $G \subset O(N)$; we will say that a function $f$ (defined on a ball, or on $\Sp$) is $G$-symmetric if $f \circ g = f$ for any $g \in G$. In this paper we will consider only functions which are either radially symmetric or $G$-symmetric.

\medskip

Denote by $\lambda_i$ the eigenvalues of the laplacian operator on the unit ball $B$ with Dirichlet boundary conditions for $G$-symmetric functions, counted with multiplicity. We also denote by $\bar{\lambda}_i$ the Dirichlet eigenvalues associated to radial eigenfunctions. Of course we have

$$  0<\lambda_1 < \lambda_2 \leq \lambda_3 \dots, \quad \ 0< \bar{\lambda}_1 < \bar{\lambda}_2 < \bar{\lambda}_3\dots $$
$$ \lambda_1 = \bar{\lambda}_1, \ \lambda_i \leq \bar{\lambda}_i \ \mbox{ if } i >1.$$

In an analogous way, we denote by $\sigma_i$ and $\bar{\sigma}_i$ the eigenvalues of the laplacian operator on the unit ball with Neumann boundary conditions for $G$-symmetric and radial functions, respectively. There holds:

$$  0=\sigma_0<\sigma_1 \leq \sigma_2 \leq \sigma_3 \dots, \quad \ 0=\bar{\sigma_0}< \bar{\sigma}_1 < \bar{\sigma}_2 < \bar{\sigma}_3\dots $$
$$ \ \sigma_i \leq \bar{\sigma}_i \ \mbox{ if } i \geq 1.$$
It is known that:

$$ 0 = \bar{\sigma}_0 < \bar{\lambda}_1< \bar{\sigma}_1 < \bar{\lambda}_2<\bar{\sigma}_2 < \bar{\lambda}_3 \dots.$$

From now on, we will fix a symmetry group $G$ of $\mathbb{S}^{N-1}$ with the following property:
\begin{enumerate}[label=(G), ref=(G)]
	\item \label{G} Denote by $\sigma$ the first eigenvalue $\sigma_k$ with $\sigma \neq \bar{\sigma}_j$ for all $j \in \N$. We assume that $\sigma> \bar{\lambda}_2$ and has odd multiplicity.
\end{enumerate}
In other words, we impose that the first nonradial $G$-symmetric Neumann eigenfunction appear at an eigenvalue higher than $\bar{\lambda}_2$, and with odd multiplicity. We will show in the Appendix that symmetry groups $G$ satisfying this property are, for instance,
\begin{itemize}
	\item If $N=2$, the dihedral group $\mathbb{D}_k$ with $k \geq 5$.
	\item If $N=3$, the group $G$ of all isometries of the icosahedron.
	\item If $N=4$, the group $G$ of all rotations of the hyper-icosahedron.
\end{itemize}

\medskip 

Let us recall that the nontrivial eigenvalues of the Laplace-Beltrami operator on $\mathbb{S}^{N-1}$ have the expression $i(i+N-2)$, $i \in \N \cup $.
For later use, we denote by $\gamma_k$ the eigenvalues of the Laplace-Beltrami operator on $\mathbb{S}^{N-1}$ for $G$-symmetric eigenfunctions, counted with multiplicity. Obviously, $\gamma_k= i(i+N-2)$ for some $i = i(k) \in \N$. 
The associated $G$-symmetric eigenfunctions are denoted by $\vartheta_k$, and are normalized so that:
$$\int_{\Sp} \vartheta_k(\theta)^2 =1.$$

\medskip

We define the spaces of H\"{o}lder continuous functions:
\begin{align*}
	C^{k,\alpha}_{G}(B(R))&=\{u\in C^{k,\alpha}(B(R)): \ \ u \mbox{ is G-symmetric} \},
\end{align*}
\begin{align*}
	C^{k,\alpha}_{G,0}(B(R))&=\{u\in C^{k,\alpha}_G (B(R)): \ \ u=0 \mbox{ on } \partial B\},
\end{align*}
\begin{align*}
	C^{k,\alpha}_{G}(\mathbb{S}^{N-1})&=\{u\in C^{k,\alpha}(\mathbb{S}^{N-1}): \ u \mbox{ is G-symmetric}\}.
\end{align*}
We also define the set of functions in $C^{k,\alpha}_{G}(\Sp)$ whose mean is $0$:
\[ C^{k,\alpha}_{G,m}(\mathbb{S}^{N-1})= \left \{u\in C^{k,\alpha}_G(\mathbb{S}^{N-1}):\ \int_{\Sp} u=0 \right \}.	 \]

\medskip

In this paper we will prove the following result:
\begin{theorem}\label{main} Let $G \subset O(N)$ a symmetry group satisfying \ref{G}. Then there exists $\tilde{\rho}>0$, and sequences $\rho_n \to \tilde{\rho}$ and $w_n \in   C^{2,\alpha}_{G,m}(\mathbb{S}^{N-1})$, $w_n \to 0$, $w_n \neq 0$ such that the problem:
	\begin{equation} \label{eq}
		\begin{cases}
			-\rho_n \Delta u = u-(u^+)^3  &\mbox{in $\Omega_n$, }\\
			u=0 &\mbox{on $\partial\Omega_n$, }\\
			\partial_{\nu} u=\mbox{constant} &\mbox{on $\partial\Omega_n$, }
		\end{cases}
	\end{equation}
admits a sign-changing G-symmetric $C^2$ solution. Here $$\Omega_n = \{x \in \R^N \setminus\{0\}:\ |x| < 1 + w_n \big(\frac{x}{|x|} \big) \} \cup \{0\}.$$
\end{theorem}
Theorem \ref{0main} follows from the previous one by a change of variables $x \mapsto \sqrt{\rho_n}x$. As commented in the introduction, Theorem \ref{main} holds true in any dimension provided that there exists a symmetry group $G$ satisfying \ref{G}. This is the case in dimensions 2, 3, and 4 at least. See the Appendix for more details.

We will need to make use of the Sobolev spaces of a domain $B(R)$ which are radially symmetric, or $G$-symmetric, which will be denoted as:

$$ H_{r}^s(B(R))= \{u \in H^s(B(R)):\  u \mbox{ is radially symmetric} \},$$
$$ H_{0,r}^s(B(R)) =\{u \in H_0^s(B(R)):\  u \mbox{ is radially symmetric} \},$$
$$ H_{G}^s(B(R))= \{u \in H^s(B(R)):\  u \mbox{ is G-symmetric} \},$$
$$ H_{0,G}^s(B(R)) =\{u \in H_0^s(B(R)):\  u \mbox{ is G-symmetric} \}.$$

Similar notation is used for Sobolev spaces on $\Sp$; for instance, we shall write $H_G^s(\Sp)$ to denote the Sobolev space of G-symmetric functions defined on the sphere.

In what follows we will make use of spaces of functions that are $L^2$-orthogonal to all radial functions. Being more specific, we define:

\begin{equation} \label{defE} E= \{ \phi \in H_{G}^1(B): \ \int_{B} \phi (x) g(x) \, dx=0 \ \ \forall \ g \in L^2_r(B) \}, \end{equation}

\begin{equation} \label{defE0} E_0= \{ \phi \in H_{0,G}^1(B): \ \int_{B} \phi (x) g(x) \, dx=0 \ \ \forall \ g \in L^2_r(B) \}. \end{equation}

We have the following characterization of the spaces $E$ and $E_0$.

\begin{lemma}\label{E} The following statements are equivalent:
	
	\begin{enumerate}
		\item[i)] $\phi \in E$ (alternatively, $\phi \in E_0$).
		\item[ii)] $\phi \in H_{G}^1(B)$ (alternatively, $\phi \in H_{0,G}^1(B)$) can be written as:
		$$ \phi = \sum_{k=1}^{+\infty} f_k(r) \vartheta_k(\theta).$$
		Here $(r, \theta) \in (0,1] \times \Sp$ are spherical coordinates and $\vartheta_k$ denote $G$-symmetric spherical harmonics. Observe that the $0$ order term is missing in the expansion.
		\item[iii)] $\phi \in H_{G}^1(B)$ (alternatively, $\phi \in H_{0,G}^1(B)$) satisfies that, for any $r \in (0,1]$,
		$$ \int_{\Sr} \phi (x) \, dx=0.$$

	\end{enumerate}

\end{lemma}

\begin{proof}
	
	We begin by showing $i) \Rightarrow ii)$. If $\phi \in H^1_{G}(B)$, it is well known that it can be expanded via Fourier series in spherical coordinates $(r, \theta) \in (0,1) \times \Sp$:
	
	$$\phi(x) = \frac{1}{\sqrt{\o_N}} \, f_0(r)  + \sum_{k=1}^{+\infty} f_k(r) \vartheta_k(\theta), \ \ \ \o_N =|\Sp|.$$
	
	Moreover,
	$$ \| \phi \|_{H^1}^2 = \sum_{k=1}^{\infty} \int_0^1 r^{N-1} \Big (f_k'(r)^2 + f_k(r)^2  +  \frac{\gamma_k^2}{r^2} f_k(r)^2\Big )\, dr.$$
	
	In particular $f_0(x) = f_0(|x|)$ belongs to $H_r^1(B)$. Taking any $g \in L^2_r(B)$, we have that by Fubini Theorem,
	
	\begin{align*}  0& =\int_B \phi(x) g(x) \, dx = \int_0^1 g(r) \left [ f_0(r)  + \sum_{k=1}^{+\infty} f_k(r) \int_{\Sr} \vartheta_k(\theta) \, d \theta \right ] \, dr \\ & = \int_0^1 g(r) f_0(r) \, dr.\end{align*}
	
Since $g$ is arbitrary, we conclude that $f_0=0$, and $ii)$ holds.
	
	\medskip 
	
	The proof of  $ii) \Rightarrow iii)$ is immediate, since
	$$ 	 \int_{\Sr} \phi (x) \, dx= \sum_{k=1}^{+\infty} f_k(r) \int_{\Sr} \vartheta_k(\theta) \, d \theta=0.$$
	
	\medskip 
	
	Finally, the proof of $iii) \Rightarrow i)$ follows from the Fubini theorem, since for any $g \in L_r^2(B)$,
	$$ \int_{B} \phi(x) g(x) \, dx = \int_0^1 g(r) \Big (\int_{\Sr} \phi(x) \Big) \, dr =0.$$

\end{proof}

\section{The family of radial solutions}

In this section we build a family of sign-changing radially symmetric solutions to the problem:
\begin{equation}\label{ball}
	\begin{cases}
		- \rho \, \Delta u = u - (u^+)^3 & \mbox{ in }B, \\
		u=0 &\mbox{on }\partial B. \end{cases}
\end{equation}
Here $\rho$ is a real positive parameter in a interval to be specified. We are also concerned with the properties of this family such as smoothness and the behavior of the linearized operator under homogeneous Dirichlet boundary conditions. This family of solutions will be used to find nonradial solutions to \eqref{eq} as a local bifurcation. 

\medskip Sometimes it will be useful to use the change of variables:

\begin{equation} \label{change} v_R(r) = u_\rho \big (\sqrt{\rho} \, r \big ), \ \ R= \rho^{-1/2}. \end{equation} In this way, equation \eqref{ball} is equivalent to:

\begin{equation}\label{ball2}
	\begin{cases}
		- \Delta v = v - (v^+)^3 & \mbox{ in }B(R), \\
		u=0 &\mbox{on }\partial B(R). \end{cases}
\end{equation}

In next proposition we show the existence of solutions to \eqref{ball} and their basic properties. Recall the definition of $\bar{\lambda}_2$ given at the beginning of Section 2.

\begin{proposition} \label{radial} For any $\rho \in (0, \bar{\lambda}_2^{-1})$ there exists a radial sign-changing $C^{4,\alpha}$ solution $u_\rho$ of \eqref{ball} satisfying that:
	
\begin{enumerate}
	\item[a)] The function $u_\rho(r)$ has a unique zero at a point $p_\rho \in (0,1)$ and 
	$$ \left \{  \begin{array}{ll}  u_\rho(r)>0 & \mbox{ for } r \in [0, p_\rho), \\ u_\rho(r)<0 & \mbox{ for } r \in (p_\rho,1). \end{array} \right.$$
	Moreover, 
	\begin{equation} \label{constant} c_\rho=u_\rho'(1) >0. \end{equation}
	
	\item[b)] The radial solution of \eqref{ball} satisfying a) is unique.

	\item[c)] If $\rho \to \bar{\lambda}_2^{-1}$, $ u_{\rho} \to 0$ in $C^4$ norm 
	\item[d)] If $\rho \to 0$, then $u_\rho \to 1$ in compact sets of $B$, in $C^4$ sense. Moreover, if we make the change of variables \eqref{change} and define: 
	\begin{equation}  \tilde{v}_R(r) = v_R (r - p_R), \ \ p_R=  \rho^{-1/2} p_\rho, \end{equation} 
	we have that $ \tilde{v}_R \to \tilde{v}_0$ in $C^4$ sense in compact sets of $(-\infty, \pi)$, where \begin{equation} \label{v0} \tilde{v}_0(r)= \left \{ \begin{array}{ll} - \tanh(\frac{r}{\sqrt{2}}) & r \leq 0, \\ - \frac{1}{\sqrt{2}} \sin(r) & r \in (0,\pi]. \end{array} \right. \end{equation}
	
\end{enumerate}

\end{proposition}

\begin{figure}[h]
	\centering 
	\begin{minipage}[c]{100mm}
		\centering
		\resizebox{100mm}{55mm}{\includegraphics{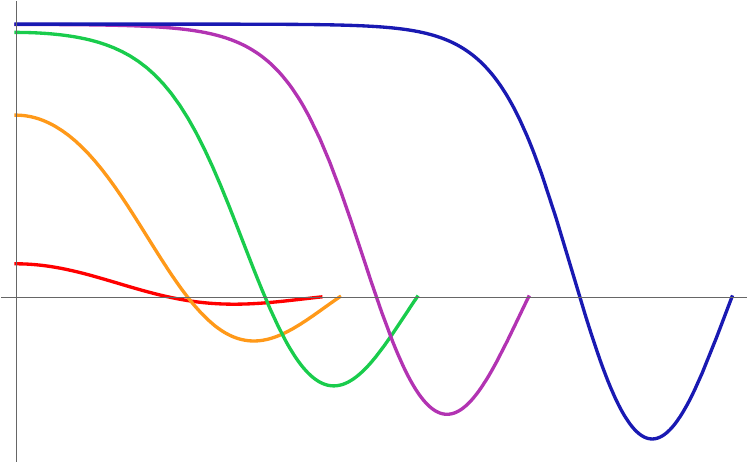}}
	\end{minipage}
	\caption{The numerical plot of the functions $v_R$ for different values of $R$ in dimension 3, using Mathematica.}
\end{figure}

\begin{proof} For the proof of a), let us perform the change of variables \eqref{change}. If $R= \bar{\lambda}_2^{1/2}$, $1$ is the second eigenvalue for radial functions of the Laplacian on $B(R)$ with Dirichlet boundary conditions; let us denote $\bar{\psi}_2(r)$ the corresponding eigenfunction. Moreover, its unique zero is achieved at $\bar{\lambda}_1^{1/2}$.

Let us now consider $R > \bar{\lambda}_2^{1/2}$, and the initial value problem:
\begin{equation*}
	\begin{cases}
		- w''(r) -(N-1) \frac{w'(r)}{r} = w(r) , \\
		w(R)=0, w'(R)=1.\end{cases}
\end{equation*}
By Sturm separation of zeroes applied to $w$ and $\bar{\psi}_2$, we conclude that $w$ becomes zero at a unique point $p_R \in (\bar{\lambda}_1^{1/2}, R)$. We emphasize that the first eigenvalue of the Laplace operator with Dirichlet boundary conditions in $B(p_R)$ is strictly smaller than 1.

We now study the energy functional:

$$F: H_0^1(B(p_R)) \to \R, \ F(z)= \int_{B(p_R)} |\nabla z|^2 - z^2 + \frac{1}{2} z^4.$$

It is straightforward to prove that $F$ is weak lower semi-continous and coercive, and then it achieves a minimum. Of course $0$ is a critical point of $F$ but it does not correspond to a local minimum since $p_R>\bar{\lambda}_1^{1/2}$. As a consequence, the minimizer of $F$ is nontrivial; let us denote it by $z$. Observe now that $|z|$ is also a minimizer, and hence a solution. By the maximum principle, we can assume that $z$ is strictly positive. By Hopf lemma, $z'(p_R)<0$.

We can then define $v_R(r)$ as follows:

$$ v_R(r)= \left \{ \begin{array}{ll} z(r) & r \in [0, p_R], \\ t w(r) & r \in (p_R, R] \end{array} \right.$$

Here $t>0$ is chosen such that $v_R(r)$ is $C^1$ at $p_R$. The $C^{4, \alpha}$ regularity follows from Schauder estimates. Observe that 

$$ p_R= R p_\rho= \rho^{-1/2} p_\rho.$$

\medskip For the proof of b), assume that $u_1$, $u_2$ are two radial solutions of \eqref{ball} satisfying a). By the change of variables \eqref{change}, we are led with two solutions $v_1$ and $v_2$ of \eqref{ball2}. By a), there exist $p_1$ and $p_2$ such that $v_i(r)<0$ for $r \in (p_i,R)$, $i=1,\ 2$. Recall that \eqref{ball2} is linear for $v<0$: then, $p_1=p_2=p_R$ and $v_1 = \alpha v_2$ for any $r >p_R$ and some $\alpha >0$. In order to conclude it suffices to show that the positive solution of:

\begin{equation*}
	\begin{cases}
		- \Delta v = v - v^3 & \mbox{ in }B(p_R), \\
		u=0 &\mbox{on }\partial B(p_R), \end{cases}
\end{equation*}
is unique. But this is known, even for nonsymmetric solutions and domains, see for instance \cite{bo}.

\medskip

The proof of c) and d) need some preliminaries. We first show that $\|v\|_{C^{4,\gamma}}$ is uniformly bounded, where $v$ is any solution to \eqref{ball2}. To prove that, first observe that $v(r) <1$ for any $r$ by the maximum principle. Observe moreover that:
$$H(r) = v'(r)^2 + v(r)^2 - \frac 1 2 (v(r)^+)^4$$
is decreasing in $r$ for any solution of \eqref{ball2}. Evaluating $H(0)$ and $H(r)$ we obtain that, if $v(r)<0$,
$$ H(r)= v'(r)^2 + v(r)^2 \leq H(0) < \frac{1}{2}.$$
Once we have boundedness in $L^{\infty}$, the $C^{4,\gamma}$ bounds follow from standard elliptic local regularity estimates.

\medskip Observe moreover that since ${v}_R$ solves a linear equation for $r>p_R$, it follows that $R - p_R$ is bounded. This follows from the uniform separation of consecutive zeroes of Bessel functions. In particular, $p_R \to +\infty$ as $R \to +\infty$. 

\medskip 
We now prove c), so consider $\rho > \bar{\lambda}_2^{-1}$, $\rho \to \bar{\lambda}_2^{-1}$. By the above discussion we can assume that $u_\rho \to u_0 $ in $C^4$ sense, up to a subsequence. Moreover, $u_0$ is a solution of
 \begin{equation*}
 	\begin{cases}
 		- \Delta u_0 = \bar{\lambda}_2 u_0 - \bar{\lambda}_2  (u_0^+)^3 & \mbox{ in }B, \\
 		u=0 &\mbox{on }\partial B, \end{cases}
 \end{equation*}
 
We now distinguish two cases. If $u_0^+ =0$, then $u_0$ is a nonnegative solution of the linear problem $	- \Delta u_0 = \bar{\lambda}_2 u_0$, which implies that $u_0=0$.

Assume now that $u_0^+$ is not trivial. We point out that $p_\rho$ does not approach $1$, and then $u_0(r)$ vanishes at a point $p \in (0,1)$, limit of the sequence $p_\rho$. Then $u_0$ is a non-principal eigenfunction of the operator $-\Delta - \bar{\lambda}_2 + \bar{\lambda}_2(u_0^+)^2$ with $0$ eigenvalue. However, $0$ is the first non-principal eigenvalue of the operator $-\Delta - \bar{\lambda}_2 $, and this is impossible by the monotonicity of the eigenvalues with respect to the potential.

%

%
%

\medskip We now turn our attention to the statement d). Define:

$$ \tilde{v}_R(r)=  v_R(\cdot - p_R).$$

Observe that $\tilde{v}_R$ is uniformly bounded in $C^{4,\gamma}$ sense. Moreover, by the maximum principle it is decreasing for $r\in (-p_R,0)$, and it solves:

\begin{equation*}
	\begin{cases}
		- \tilde{v}_R''(r) -(N-1)\,  \displaystyle \frac{\tilde{v}_R'(r)}{r + p_R} = \tilde{v}_R(r) - (\tilde{v}_R^+(r))^3 , \\
		\tilde{v}_R(0)=0.\end{cases}
\end{equation*}

Then, in compact sets of $(-\infty, \pi)$ the sequence $\tilde{v}_R$ converges to $\tilde{v}_0$ in $C^4$ sense to the solution of the problem:

\begin{equation*}
	\begin{cases}
		- \tilde{v}_0''(r) = \tilde{v}_0(r) - (\tilde{v}_0^+(r))^3 , r \in (-\infty, \pi)\\
	\tilde{v}_0(0)=0.\end{cases}
\end{equation*}

Moreover $\tilde{v}_0$ is decreasing for all $r<0$, and it is well known that necessarily $\tilde{v}_0(r)=- \tanh(\frac{r}{\sqrt{2}})$ for $r<0$. Then we conclude that $\tilde{v}_0$ is as in \eqref{v0}. The fact that $u_\rho \to 1$ in compact sets of $B$ follows from the above description and since $\tilde{v}_R$ is decreasing for $r<0$.

\end{proof}

In the following proposition we study the linearized operator of the solution $u_\rho$ under Dirichlet boundary conditions. We will consider both the radially symmetric and the $G$-symmetric case. We denote by $L:H_{0,G}^1(B) \to (H_{0,G}^1(B))^{-1}$ the linearized operator:
\begin{equation} \label{L}  L= - \rho \Delta - 1 + 3(u^+_{\rho})^2.\end{equation}
We will consider the eigenvalue problems:

\begin{equation*} L (\varphi) = \mu_k \varphi, \ \varphi \in H_{0,G}^1(B), \end{equation*}

\begin{equation*}  L (\varphi) = \bar{\mu}_k \varphi, \ \varphi \in H_{0,r}^1(B). \end{equation*}

We also denote by $\varphi_k$, $\bar{\varphi}_k$ the associated eigenfunctions. Clearly, $\mu_1 = \bar{\mu}_1$ is simple, and $\varphi_1$ is positive. We also define the quadratic form associated $Q_D:H_{0,G}^1(B) \to \R,$
\begin{equation} \label{QD} Q_D(\phi) = \int_{B} \rho |\nabla \phi|^2 - \phi^2 + 3 (u_{\rho}^+)^2 \phi^2.\end{equation}


We denote by $\bar{Q}_D$ the restriction of $Q$ to radial functions:
$$\bar{Q}_D = Q_D|_{H_{0,r}^1(B)}.$$ 

\begin{proposition}\label{linear1} The following holds:
\begin{enumerate}
	\item[a)] For any $\rho \in (0, \bar{\lambda}_2^{-1})$ the function $u_\rho$ is nondegenerate and has Morse index 1 in $H_{0,r}^1(B)$. In other words, $\bar{\mu}_1 <0 < \bar{\mu}_2$. 
	
	\item[b)] The map $U: (0, \bar{\lambda}_2^{-1}) \to C^4(B)$, $U(\rho)= u_\rho$ is $C^1$.
	
	\item[c)] There exists $\e>0$ such that for $\rho \in (\bar{\lambda}_2^{-1}-\e,\bar{\lambda}_2^{-1})$, $\mu_2>0$.
	
	\item[d)] There exists $\e>0$ such that for $\rho \in (0, \e)$, $\mu_2<0$.

\end{enumerate}	

\end{proposition}

\begin{proof} In order to prove a), multiply \eqref{ball} by $u_{\rho}^-$ and integrate to obtain that:
$$\bar{Q}_D(u_\rho^-)=0.$$

If $\bar{Q}_D$ were semipositive definite, then $u_{\rho}^-$ would be a minimizer and hence a solution of the linearized equation. But $u_\rho^-$ is not $C^1$ regular, and this gives us a contradiction. Then the linearized operator $L$ has a negative eigenvalue $\bar{\mu}_1<0$.

Assume now that $\bar{\mu}_2 \leq 0$. Then, $\bar{Q}_D (\psi) \leq 0$ for any $\phi \in span\{ \bar{\varphi}_1, \bar{\varphi}_2\}$. Now, take $\alpha$, $\beta \in \R$ such that $\phi = \alpha \bar{\varphi}_1 + \beta \bar{\varphi}_2$ vanishes at the point $p_\rho$. We have that:

$$0\geq \bar{Q}_D (\psi)= \int_{B(p_\rho)} \rho |\nabla \phi|^2 - \phi^2 + 3 (u_{\rho}^+)^2 \phi^2 + \int_{B\setminus B(p_\rho)} \rho |\nabla \phi|^2 - \phi^2  = (I) + (II).$$

We now observe that in $B\setminus B(p_\rho)$, $u_\rho^-$ is a negative solution of the linear problem $-\Delta u + u =0$. As a consequence,
$$ (II) \geq 0, \mbox{ and } (II) =0 \Leftrightarrow \phi(r) = \alpha u_{\rho}(r)^- \mbox{ for } r\in (p_\rho,1), \ \alpha \in \R.$$

Regarding $(I)$, observe that $u_\rho|_{B_(p_\rho)}$ solves the equation $\tilde{L}(u)=0$, where $\tilde{L}= -\rho \Delta - 1 + (u_\rho^+)^2$. Since  $u_\rho|_{B(p_\rho)}>0$, we conclude that $0$ is the first eigenvalue of $\tilde{L}$ with Dirichlet boundary conditions. As a consequence, the operator $L= -\rho \Delta - 1 + 3(u_\rho^+)^2>\tilde{L}$ has only positive eigenvalues. Then,

$$ (I) \geq 0, \mbox{ and } (I) =0 \Leftrightarrow \phi = 0 \mbox{ for } r \in [0, p_\rho].$$

Since $\phi$ is a regular function it turns out that $\phi=0$, which implies that $\bar{\varphi}_1$ and $\bar{\varphi}_2$ are linearly dependent functions, a contradiction.

\medskip The proof of b) is an immediate consequence of a), Proposition \ref{radial} and the Implicit Function Theorem.

\medskip 
	To prove c) we can assume that the eigenfunction $\varphi_2$ associated to $\mu_2$ is not radially symmetric, since  otherwise we are done by a). By Fourier decomposition, we obtain that $\varphi_2(x)= f(r) \vartheta_k(\theta)$, where $f$ is a radial function with $f(1)=0$ and $\vartheta_k$ is a $G$-symmetric spherical harmonic. However, observe that  for any nontrivial function $\phi \in H^1_G(B)$ that is written as  $\phi_2(x)= f(r) \vartheta_k(\theta)$ (even if $f(1) \neq 0$), we have:
	
	\begin{equation*}  \mu_2 \int_{B} \varphi_2^2=  Q_D(\phi) = \int_{B} \rho |\nabla \phi|^2 - \phi^2 + 3 (u_{\rho}^+)^2 \phi^2 \geq \int_{B} \rho |\nabla \phi|^2 - \phi^2 .\end{equation*}

	But by our assumption \ref{G} on the symmetry group G,	
	$$ \int_{B} \bar{\l}_2^{-1} |\nabla \phi|^2 - \phi^2 >0.$$
	
	And this implies that $\mu_2>0$ if $\rho$ is sufficiently close to $\bar{\l}_2^{-1}$.
\medskip 
		
	In order to prove d) we shall make use of the change of variable \eqref{change} and consider instead the form $Q'_D:H_{0,G}^1(B(R)) \to \R$,
	
	$$Q'_D(\psi)= \int_{B(R)} |\nabla \psi|^2 - \psi^2 + 3 (v_{R}^+)^2 \psi^2.$$
	
	We consider also the limit form $\hat{Q}_D: H_0^1(-\infty, \pi)\to \R$,
	$$   \hat{Q}_D(\xi) = \int_{-\infty}^{\pi}  |\xi'(r)|^2 - \xi^2 + 3 (\tilde{v}_{0}^+)^2 \xi^2,$$
	where $\tilde{v}_0$ is given in \eqref{v0}. We take:
	
	$$ \xi_0(r)= \left \{ \begin{array}{ll} \sin(r) & r \in [0, \pi], \\ 0 & r <0. \end{array} \right.$$

	Then, $\hat{Q}_D(\xi_0)=0$. If $\hat{Q}_D $ is semipositive definite, then $\xi_0$ would be a first eigenfunction, but $\xi_0$ is not $C^1$ and this gives a contradiction. Then, $\hat{Q}_D $ attains negative values; by density, there exists $\xi \in C_0^{\infty}(-\infty, \pi)$ such that $\hat{Q}_D(\xi) <0$.
	
	Take now $\vartheta_1: \mathbb{S}^{N-1} \to \R$ the first G-symmetric eigenfunction of the Laplace-Beltrami operator on $\mathbb{S}^{N-1}$. We use spherical coordinates $\R^N \ni x = (\rho, \theta) \in (0, +\infty) \times \mathbb{S}^{N-1}$, and define the test function:
	$$ \psi(x) = \xi(r-p_R) \vartheta_k(\theta).$$
	Observe now:
	
	$$ | \nabla \psi |^2 = \xi'(r-p_R)^2\, \vt^2(\theta) + \frac{1}{r^2} \, \xi(r-p_R)^2\,   |\nabla_{\theta} \vartheta(\theta)|^2. $$

	We now estimate $Q_D'(\psi)$ for large $R$, taking into account Proposition \ref{radial}, d).

\begin{eqnarray*}Q_D'(\psi)= \int_{-\infty}^{\pi} \left [ (r + p_R )^{N-1} \big ( \xi'(r)^2 - \xi(r)^2 + 3 (v_R (r + p_R)^+)^2 \xi(r)^2 \big ) + \right.\\ \left.  (r + p_R)^{N-3} \gamma_k \xi(r)^2 \right ] \, dr   = p_R^{N-1}  \hat{Q}_D(\xi) + o(p_R^{N-1}) <0.\end{eqnarray*}
 
This implies the assertion d) and the proof is done.

\end{proof}

Proposition \ref{linear1} motivates the definition:

\begin{equation} \label{rho0} \rho_0 = \sup \{ \rho \in (0, \bar{\lambda}_2^{-1}): \ \mu_2(\rho) \leq 0\} \in (0, \bar{\lambda}_2^{-1}).\end{equation}

\begin{remark} \label{remark} Observe that for any $\rho \in (\rho_0, \bar{\lambda}_2^{-1})$, $\mu_2>0$ and then the linearized operator $L$ defined in \eqref{L}, acting on $H_{0,G}^1(B)$, is nondegenerate. Furthermore, the quadratic form $Q_D$ is positive definite in the space $E_0$. 
	
Finally, a continuity argument shows that $\mu_2=0$ for $\rho = \rho_0$. These facts will be essential in the rest of the paper.
\end{remark}
We finish this section with a result regarding the linearized operator with boundary data. For that, recall the definition of the space $E$ given in \eqref{defE}.

\begin{proposition} \label{psiv} Assume that $\rho \in (\rho_0, \bar{\lambda}_2^{-1})$. Then, the followings assertions hold true:
	
\begin{enumerate}
	\item[a)] For any given function $w \in H^{1/2}_G(\Sp)$ there exists a unique $\psi_{\rho,w}=\psi_w \in H^1_G(B)$ solving (in a weak sense):
\begin{equation} \label{psi}
		\begin{cases}
			- \rho \Delta \psi_w - \psi_w + 3 (u_\rho^+)^2 \psi_w=0, & \mbox{ in }B, \\
			\psi_w=w &\mbox{on }\partial B. \end{cases}
	\end{equation}
	
	\item[b)] If $w \in H^{1/2}_G(\Sp)$ and $\int_{\Sp} w =0$, then $\psi_w \in E$. Moreover, if $w \in H^{s}_G(\Sp)$ for some $s>1$,
	\begin{equation*}  \int_{\Sp} \partial_{\nu} \psi_w =0. \end{equation*}  
	
	\item[c)] If $w \in C^{2,\alpha}_G(\Sp)$ then $\psi_w \in C^{2,\alpha}_G(B)$.

	\end{enumerate}
	
\end{proposition}

\begin{proof}
	For the proof of a), take any function $\phi \in H^1_G(B)$ such that $\phi|_{\Sp} = w$, and define
	$$ \xi= 	- \rho \Delta \phi - \phi + 3 (u_\rho^+)^2 \phi \in (H_{0,G}^1(B))^{-1}.$$
	Since the operator $L$ defined in \ref{L} is an isomorphism, we can solve uniquely the problem:
	\begin{equation*} 
		\begin{cases}
			- \rho \Delta \varphi - \varphi + 3 (u_\rho^+)^2 \varphi=\xi, & \mbox{ in }B, \\
			\varphi=0 &\mbox{on }\partial B. \end{cases}
	\end{equation*}
Then $\psi= \phi - \varphi$ is the unique solution of \eqref{psi}.

We now show b). We use the Fourier expansion of $\psi_w$ as:
	
	$$\psi_w(x) = \frac{1}{\sqrt{\o_N}} \, f_0(r)  + \sum_{k=1}^{+\infty} f_k(r) \vartheta_k(\theta), \ \ \ \o_N =|\Sp|.$$
	
	Observe that $f_0$ can be recovered as:
	
	$$ f_0(r)= \frac{1}{\sqrt{\o_N} r^{N-1}} \int_{\mathbb{S}^{N-1}(r)} \psi_w = \sqrt{\o_N} \fint_{\mathbb{S}^{N-1}(r)} \psi_w.$$
	It is well known that $f_0(x)= f_0(|x|) \in H_{0,r}^1(B)$ and solves weakly the equation:
	
		\begin{equation*} 
		\begin{cases}
		- \rho \Delta f_0 - f_0 + 3 (u_\rho^+)^2 f_0=0, & \mbox{ in }B, \\
			f_0=0 &\mbox{on }\partial B. \end{cases}
	\end{equation*}

By Proposition \ref{linear1}, a), we obtain that $f_0=0$. And this implies that $\psi_w \in E$ (recall Lemma \ref{E}).
	
	\medskip 
	
	Moreover, if $w \in H^{s}(\Sp)$, then $\psi_w \in H^{s+ 1/2}(\Omega)$, and hence $\partial_\nu \psi_w$ belongs to $H^{s-1/2}(\Sp)$. Moreover, the functions $f_k$ belong to $H^{s+1/2}(1/2,1)$, and in particular they are $C^1$ close to $r=1$. And then we can compute:

$$ \partial_{\nu} \psi_w(\theta) = 	\sum_{k=1}^{+\infty} f_k'(1) \vartheta_k(\theta),$$

which has indeed $0$ mean.

\medskip The proof of c) follows from classical Schauder theory.

\end{proof}

\section{The Dirichlet-to-Neumann operator}

In this section we build the Dirichlet-to-Neumann operator to which we intend to apply a local bifurcation argument. In order to do this, some definitions are in order. First, given $w \in  C_{G}^{2,\alpha}(\mathbb{S}^{N-1})$, $w>0$, we define the radial graph

$$B(w)= \{ x \in \R^N\setminus\{0\}: |x| < w\Big(\frac{x}{|x|} \Big)\} \cup \{0\}.$$
In spherical coordinates $(r,\theta)$,
$$ B(w)= \{ (r, \theta) \in [0, +\infty) \times \Sp:\   \ r < w(\theta)\}.$$

\medskip The main result of this section is the following:

\begin{proposition} \label{F}
	For any  $\rho \in (\rho_{0}, \ \bar{\lambda}_{2}^{-1})$, there exists a neighborhood $\mathcal{O}$ of $0$ in $ C_{G}^{2,\alpha}(\mathbb{S}^{N-1})$ such that for any $w \in \mathcal{O}$, the problem
	\begin{equation} \label{eqD}
	\begin{cases}
			- \rho \, \Delta u = u - (u^+)^3 & \mbox{ in }B(1+w), \\
			u=0 &\mbox{on }\partial B(1+w). \end{cases}
	\end{equation}
	has a unique solution $u=u_{\rho, w}$ in a neighborhood of $u_\rho $ in $ C_G^{2,\alpha}(B(1+w))$. Moreover the dependence of $u$ on the function $w$ is $C^1$ and $u_{\rho, 0}=u_{\rho}$.
\end{proposition}

\begin{proof}
	Let $w\in C_{G}^{2,\alpha}(\mathbb{S}^{N-1})$. It will be more convenient to consider the fixed domain $B$ endowed with a new metric depending on $w$. On that purpose, we consider the parameterization of $B(1+w)$ defined by $\Pi: B\rightarrow  B(1+w),$
	\begin{equation*}
		\Pi(r,\theta):=\left(\left(1+\chi(r)w(\theta)\right)r,\theta\right),
	\end{equation*}
	where $\chi$ is a $C^{\infty}$ cut-off function such that:
	\begin{equation} \label{cutoff}
		\chi(r)=
		\begin{cases}
			0, &r \leq\frac{1}{3},\\
			1, &r \geq\frac{2}{3}.		\end{cases}
	\end{equation}	
Denoting by $g$ the euclidean metric in $B$ and defining:
$$ g_w = \Pi^* g, \ \hat{u}=\Pi^{*}u,$$
we can write the problem (\ref{eqD}) as:
	\begin{equation} \label{eqDv}
		\begin{cases}
			-\rho\Delta_{g_{w}} \hat{u}= \hat{u}-(\hat{u}^+)^{3}&\mbox{in $B$ },\\
			\hat{u}=0 &\mbox{on $\partial B$}.
		\end{cases}
	\end{equation}

We now define $\mathcal{B} = \{ w \in C_{G}^{2,\alpha}(\mathbb{S}^{N-1}): \| w \|_{C^{2, \alpha}} <1/2 \}$ and:
$$N: C^{2,\alpha}_{G,0}(B) \times \mathcal{B} \to C^{0, \alpha}_G(B), $$
\begin{equation*} N(\psi,w) =-\rho \Delta_{g_w}\psi -\psi +(\psi^{+})^{3}.\end{equation*}
The proof reduces then to show that the equation $N(\psi,w)$ is solvable for any $w$ in a neighborhood of $0$. Observe now that if $w\equiv0$, then $g_{0}=g$ is the euclidean metric on $B$, and $u_{\rho}$ is a solution of (\ref{eqDv}).  In other words,

\[N(u_\rho,0)=0.\]

The mapping $N$ is $C^1$ from a neighborhood of $(u_\rho,0)$ and the partial differential of $N$ with respect to $\psi$ at $(u_\rho,0)$ is
	\[D_{\psi}N(u_\rho,0)=-\rho \Delta \psi -\psi +3(u_\rho^{+})^{2}\psi.\]

By the definition of $\rho_0$ (see \eqref{rho0}), the operator $L$ defined in \eqref{L} is an isomorphism. By regularity arguments, also $D_{\psi}N(u_\rho,0)$ is an isomorphism  between $C^{2,\alpha}_{G,0}(B)$ and $C^{0,\alpha}_{G}(B)$, and Proposition \ref{F} follows from the Implicit Function Theorem.
	
\end{proof}

\begin{definition} \label{defF} By the above result, we have the existence of a neighborhood $\mathcal{U}$ of the set $\{(\rho, u_\rho), \ \rho \in (\rho_{0},\bar{\lambda}_{2}^{-1})\}$ in the space $(\rho_{0},\bar{\lambda}_{2}^{-1}) \times C^{2,\alpha}_G(\Sp) $ such that $\eqref{eqD}$ is solvable for any $(\rho, w) \in \mathcal{U}$. After the canonical identification of $\partial B(1+w)$ with $\mathbb{S}^{N-1},$ we can define: $F: \mathcal{U} \rightarrow C^{1,\alpha}_{G}(\mathbb{S}^{N-1})$,
\begin{equation}\label{eqF}
	F(\rho,w)= \partial_{\nu} (u_{\rho, w}) \Big|_{\partial B(1+w)}-\frac{1}{|\partial B(1+w)|}\int_{\partial B(1+w)}\partial_{\nu} (u_{\rho,w}).
\end{equation}

\end{definition}

 Let us point out that $F(\rho, w)=0$ if and only if $\partial_{\nu} (u_{\rho, w})$ is constant on the boundary $\partial B(1+w)$. Obviously, $F(\rho, 0)=0$ for all $\rho \in (\rho_{0},\bar{\lambda}_{2}^{-1})$. Then Theorem \ref{main} follows if we find a branch of nontrivial solutions $(\rho, w)$ to the equation $F(\rho, w)=0$ bifurcating from some point $(\tilde{\rho}, 0)$. For this aim, we will use a local bifurcation argument. This leads us to the study of the linearization of $F$ around a point $(\rho, 0)$.

\medskip

For $\rho\in (\rho_{0},\bar{\lambda}_{2}^{-1})$, we can define the linear continuous operator $H_{\rho}:C^{2,\alpha}_{G}(\mathbb{S}^{N-1})\rightarrow C_{G}^{1,\alpha}(\mathbb{S}^{N-1})$ by
\begin{equation}\label{defH}
	H_{\rho}(w) =  \partial_{\nu} (\psi_{\rho,w}) + (N-1) \, w,
\end{equation}
where $\psi_{\rho, w}$ is given by Proposition \ref{psiv}. Let us point out that if $w$ has zero mean, then $\psi_{\rho, w}$ belongs to $E$ and hence $H_\rho(w)$ also has zero mean.

\medskip

In next proposition we prove that the linearization of the operator $F$ with respect to $v$ at $w=0$ is given by $H_{\rho}$, up to a multiplicative constant.

\begin{proposition} \label{H}
	For any $\rho\in ( \rho_{0}, \, \bar{\lambda}_{2}^{-1})$ we have that for any $w \in C^{2,\alpha}_{G,m}(\Sp)$,
	\[
	D_{w} F(\rho, 0) =- c_\rho \,  H_{\rho}(w) .
	\] 
where $c_\rho$ is given in \eqref{constant}.

\end{proposition}

\begin{proof} 	By the $C^1$ regularity of $F$, it is enough to compute the directional derivative of $F$ with respect to $w$, that is,
\begin{align*}
		D_{w} F(\rho, 0)=\mathop {\lim}\limits_{s\rightarrow 0}\frac{F(\rho, sw)-F(\rho, 0)}{s}= \mathop {\lim}\limits_{s\rightarrow 0}\frac{F(\rho, sw)}{s}.
\end{align*}
	
As in the proof of Proposition \ref{F}, we make use of the parameterization $\Pi_s$ defined:
	 
\begin{equation*}
	\Pi_s(r,\theta):=\left(\left(1+\chi(r) s w(\theta)\right)r,\theta\right),
\end{equation*}
where $\chi$ is a $C^{\infty}$ cut-off function satisfying \eqref{cutoff}. Then we define:
	$$ g_s = \Pi_s^* \, g, \ \hat{u}_s=\Pi_s^{*} (u_{\rho, sw}),$$
	where $u_{\rho, sw}$ is the function given in Proposition \ref{F}. Observe now that the solution $u_\rho$ of \eqref{ball} defined in Proposition \ref{radial} can be extended to $B(1+\delta)$ for some $\delta >0$. Then, we can consider $u_\rho$ defined in $B(1+sw)$ for sufficiently small $s$ and define:
	$$ \breve{u}_s= \Pi_s^{*}u_{\rho}.$$
	
	Observe that both functions $\hat{u}$ and $\breve{u}$ depend on $s$ in a $C^1$ sense and coincide when $s=0$. They also depend on $\rho$, but in this proof we consider $\rho$ fixed. Moreover, both functions solve the equation
	
	\begin{equation*}
		-\rho \Delta_{g_{s}} z - z + (z^+)^3=0 \quad \mbox{ in } B.
	\end{equation*}
On the boundary $\partial B$ the functions $\hat{u}_s$ and $\breve{u}_s$ behave differently:
	
$$ \hat{u}_s(\theta)=0, \ \breve{u}_s(\theta)= u_{\rho}(1+sw(\theta)), \ \theta \in \Sp.$$ 
	
Let $\hat{u}_s= \psi_s + \breve{u}_s$; then $\psi_s$ satisfies:
	\begin{equation}\label{psipunto}
	\begin{cases}
		-\rho \Delta_{g_{s}} \psi_s - \psi_s + [(\psi_s + \breve{u}_s)^+]^3 - [(\breve{u}_s)^+]^3  =0  &\mbox{in } B,\\
		\psi=-u_{\rho}(1+sw(\theta)) &\mbox{on } \partial B.
	\end{cases}
\end{equation}
 Clearly the function $\psi_s$ is differentiable with respecto to $s$ and it vanishes for $s=0$. We set
 \[\dot{\psi}=\partial_{s} (\psi_s)|_{s=0}.\]
 Differentiating (\ref{psipunto}) with respect to $s$ at $s=0$, we get that
	\begin{equation*}
	\begin{cases}
		-\rho \Delta_{g} \dot{\psi} - \psi + 3 (u_\rho^+)^2 \dot{\psi}  =0  &\mbox{in } B,\\
		\dot{\psi}=- \partial_r u_\rho(1) w  &\mbox{on } \partial B.
	\end{cases} \end{equation*}
Then $\dot{\psi} = - c_\rho \, \psi$ where $\psi= \psi_{\rho, w}$ is as given in Proposition \ref{psiv}.  

Then, we can write
\[\hat{u}_s(r,\theta)=\breve{u}_s(r,\theta) - s c_\rho \, \psi(r,\theta)+ O(s^{2}).\] 

In a neighborhood of $\partial B$ the cut-off function $\chi$ is constantly equal to 1, and then we have:
\begin{align*}
	\hat{u}_s(r,\theta) & =u_\rho((1+s w(\theta))r) - s c_\rho \, \psi(r,\theta)+ O(s^{2})\\
	&=u_{\rho}(r)+s\big[ r w(\theta)\partial_{r}u_{\rho}(r) -  c_\rho \, \psi(r,\theta) \big]+ O(s^{2}).
\end{align*}

In order to complete the proof of the result, it is enough to compute the derivative of the function $\hat{u}_s$ along the normal vector to $\partial B$ with respect to the metric $g_{s}$. In order to do that, we need the explicit expresion of $g_s$ in a neighborhood of $\partial B$:
	
\[g_{s}=a^{2}dr^{2}+\sum_{i=1}^{N-1} a\, b_i \, dr\, d\theta_i +\sum_{i=1}^{N-1} b_i^{2}+r^{2}\left(\left(1+\chi(r)v(\theta)\right)r\right)]\, d \theta_i^2 ,\]
where 
$$a=1+sw(\theta),\ b_i = s \frac{dw}{d\theta_i} r.$$

It follows from this expression that the unit normal vector field to $\partial B$ for the metric $g_{s}$ is given by
	\[\nu_s=\big((1+s w)^{-1}+{O}(s^{2})\big)\partial_{r}+{O}(s)\partial_{\theta_{i}}= \big((1-s w)+{O}(s^{2})\big)\partial_{r}+{O}(s)\partial_{\theta_{i}} ,\]
where $\theta_{i}$ are the vector fields induced by a parameterization of $\mathbb{S}^{N-1}$. As a result,
	\begin{align*} \frac{\partial \hat{u}_s}{\partial \nu_s} &  = (1-sw(\theta)) \partial_r \Big \{ u_{\rho}(r)+s\big[ r w(\theta)\partial_{r}u_{\rho}(r) -  c_\rho \, \psi(r,\theta) \big] \Big \} \Big |_{r=1}+ O(s^{2}) \\ & =
	\partial_r u_{\rho}(1) + s \Big \{  w \, \partial_{rr }u_{\rho}(1) - c_\rho \psi_r(1,\theta) \Big \} + O(s^2).	\end{align*}
	
Recall now that $c_\rho = \partial_r u_\rho(1)$; moreover, $u_\rho$ solves \eqref{ball} and then, $\partial_{rr} u_\rho	(1)= -(N-1)u_\rho(1)= -(N-1) c_\rho$. Then,
\[  \frac{\partial \hat{u}_s}{\partial \nu_s}  =  	\partial_r u_{\rho}(1) + s \, c_\rho \Big \{-(N-1) w(\theta)  - \psi_r(1,\theta) \Big \} + O(s^2).\]
Taking into account that $\partial_r u_\rho$ is constant on $\partial B$, the result follows.
\end{proof}

\section{The behavior of the linearized operator}

In this section we study the behavior of the first eigenvalue of the operator $H_\rho$ defined in \eqref{defH} depending on the parameter $\rho$. In particular, we will show that this eigenvalue changes sign, which will be the key to show bifurcation.

\medskip

Let us define the quadratic form associated to $H_{\rho}$, namely:
\[ J_\rho: C^{2,\alpha}_{G,m}(\mathbb{S}^{N-1}) \to \R, \ J_\rho (w) = \int_{\mathbb{S}^{N-1}}w H_\rho(w). \]
We denote the first eigenvalue of the operator $H_\rho$ as
\[\tau_{1}(H_\rho)=\inf \Big\{J_\rho(w): w\in C^{2,\alpha}_{G,m}(\mathbb{S}^{N-1})~,~~\int_{\mathbb{S}^{N-1}}w^{2}=1\Big \}.\]
By the divergence formula, we get
\begin{equation} \label{QJ} J_\rho (w)=\frac{1}{\rho}Q_\rho(\psi_{\rho, w}),\end{equation}
where $\psi_{\rho,w }$ is given in \eqref{psi} and $Q_\rho:H_{G}^1(B) \to \R$ is defined as
\begin{equation} \label{Q} Q_\rho(\phi) = \int_{B}\Big ( \rho |\nabla \phi|^2 - \phi^2 + 3 (u_{\rho}^+)^2 \phi^2 \Big )+ (N-1) \rho \int_{\partial B} \phi^2.\end{equation}

Observe that $Q|_{H_{0,G}^1(B)} = Q_D$ as defined in \eqref{QD}. In most cases we will consider $Q_\rho$ restricted to the space $E$ defined in \eqref{E}. In the next lemma we establish a connection between the quadratic forms $J_\rho$ and $Q_\rho$.

\begin{lemma} \label{tau}	For any $\rho\in(\rho_{0}, \bar{\lambda}_2^{-1})$, we have
	\[\tau_{1}(H_\rho)= \min \Bigg\{\frac{1}{\rho}Q_\rho(\psi): \psi\in E ,~\int_{\partial B}\psi^{2}=1\Bigg\}.\]  Moreover the infimum is attained.
\end{lemma}
\begin{proof}
	Let us define
	\begin{equation*}
		\eta:=\inf\Bigg\{Q_\rho(\psi):\psi\in E,\int_{\partial B}\psi^{2}= 1 \Bigg\} \in [-\infty, +\infty).
	\end{equation*}

	We first show that $\eta$ is achieved. On that purpose, let us take $\psi_{n}\in E$ such that $Q_\rho^D(\psi_n)\rightarrow\eta .$
	We claim that $\psi_n$ is bounded. Reasoning by contradiction, if $\|\psi_n\|_{H_{G}^1}\rightarrow +\infty$, we define $\xi_n=\|\psi_n\|_{H^1}^{-1}\psi_n$. Up to a subsequence, we can assume that $\xi_n\rightharpoonup \xi_{0}.$ Notice that $\int_{\partial B_1}\xi_n^{2}\rightarrow0,$ which yields that $\xi_{0}\in H^{1}_{0,G}(B).$ We also point out that by compactness,
	\[\int_{B} (u_\rho^+)^{2}\xi_n^{2}\rightarrow\int_{B }(u_\rho^+)^{2}\xi_{0}^{2}, \ \ 	\int_{B} \xi_n^{2}\rightarrow\int_{B }\xi_{0}^{2} .\]
	Let us distinguish two cases:\\
	\textbf{Case 1:} $\xi_{0}=0.$ In this case
	\[Q_\rho(\psi_n)=\|\psi_n\|_{H^1}^{2}\int_{B}\big(\rho|\nabla\xi_n|^{2} -\xi_n^{2} + 3 (u_\rho^+)^{2}\xi_n^{2}\big) + (N-1) \rho \rightarrow+\infty\,,\]
	which is impossible.\\
	\textbf{Case 2:} $\xi_{0}\neq 0.$ Then,
	\begin{align*}
		\mathop {\liminf}\limits_{m\rightarrow\infty}Q_\rho(\psi_n)&=\mathop {\liminf}\limits_{m\rightarrow\infty}
		\|\psi_n\|_{H^1}^{2}\int_{B}\big(\rho|\nabla\xi_n|^{2} -\xi_n^{2} + 3 (u_\rho^+)^{2}\xi_n^{2}\big) + (N-1) \rho \\
		&\geq\mathop {\liminf}\limits_{m\rightarrow\infty}\|\psi_n\|_{H^1}^{2} Q_\rho(\xi_{0})\,,
	\end{align*}
	but $Q_\rho(\xi_{0})=Q_\rho^D(\xi_{0})>0$ for $\rho\in(\rho_{0}, \bar{\lambda}_2^{-1})$, by the definition of $\rho_0$ (see Remark \ref{remark}). This is again a contradiction.
	
	Thus, $\psi_n$ is bounded, so up to a subsequence we can pass to the weak limit $\psi_n\rightharpoonup\psi.$ By compactness, $\int_{\Sp} \psi^2=1$, and then $\psi$ is a minimizer for $\eta$. In particular $\eta>-\infty.$
	
	\medskip
	
	Since $\psi$ is a minimizer, by the Lagrange multiplier rule, there exist $\kappa \in \R$  so that for any $g \in E$,
	\begin{equation} \label{weak2} \int_{B}\Big ( \rho \nabla \psi \cdot \nabla g - \psi g  + 3 (u_{\rho}^+)^2 \psi g \Big )+ (N-1) \rho \int_{\partial B} \psi g = \kappa \int_{\partial B} \psi g. \end{equation}
	
	By testing $g= \psi$ in the identity above we conclude that $\kappa = \eta$. Now we claim that \eqref{weak2} holds also if $g \in H^1_r(B)$. By density, it suffices to take a radial function $g \in C^2(\overline{B})$. Observe that, by Lemma \ref{E},
	$$ \int_{B}  \psi g  + 3 (u_{\rho}^+)^2 \psi g =0, \ \ \int_{\partial B} \psi g =0.$$
	
	Moreover, if $g$ is radially symmetric, also $\Delta g$ is radial and then
	
	$$ \int_{B} \rho \nabla \psi \cdot \nabla g = \int_{\partial B} \psi \partial_{\nu} g - \int_B \Delta g \psi =0,$$
	again by Lemma \ref{E}. 
	
Then we have shown that \eqref{weak2} holds for any $g \in H^1_{G}(B)$. In other words, $\psi$ is a (weak) solution of the problem:
\begin{equation} \label{sol}
	\begin{cases}
		- \rho \Delta \psi - \psi + 3 (u_\rho^+)^2 \psi=0, & \mbox{ in }B, \\
		\partial_{\nu}\psi + (N-1) \psi= \eta \, \psi. &\mbox{on }\partial B. \end{cases}
\end{equation}

By the regularity theory, $\psi\in C^{2,\alpha}_{G}(B)$. Define $w = \psi|_{\partial B}$; clearly, $$\int_{\partial B} w^2=1$$ and $\psi= \psi_w$ as defined in Proposition \ref{psiv}. Then:
	$$ J_\rho(w)= \frac{1}{\rho} Q_\rho(\psi)= \frac{1}{\rho} \eta\,,$$
which implies that $\tau_1(H_\rho) \leq \frac{1}{\rho} \eta$. 

But the other inequality $\tau_1(H_\rho) \geq \frac{1}{\rho} \eta$ follows directly form the relation \eqref{QJ}. The proof is completed.
\end{proof}

Taking into account the previous result, we are now concerned with the behavior of the quadratic form $Q$ restricted to the space $E$ defined in \eqref{E}. This is the purpose of the following result.

\begin{proposition} \label{linear2} The following assertions hold true:
\begin{enumerate}	\item[a)] There exists $\e>0$ such that for $\rho \in (\bar{\lambda}_2^{-1}-\e,\bar{\lambda}_2^{-1})$, $Q|_E$ is positive definite.

\item[b)] There exists $\e>0$ such that for any $\rho \in (\rho_0, \rho_0+\e)$ there exists some $\psi \in E$ such that $Q(\psi)<0$.

\end{enumerate}	

\end{proposition}

\begin{proof}
The proof of a) is basically a repetition of the argument to prove Proposition \ref{linear1}, c). We consider the quadratic form associated to homogeneous Neumann boundary conditions $Q_N:H_{G}^1(B) \to \R,$
\begin{equation*} Q_N(\phi) = \int_{B} \rho |\nabla \phi|^2 - \phi^2.\end{equation*}

Since $Q \geq Q_N$, it suffices to show that $Q_N$ is positive definite in $E$ for $\rho$ close to $\bar{\l}_2^{-1}$.
But this is immediate from the assumptions on the symmetry gropu $G$, since the first nonradial eigenvalue of the Laplacian operator with Neumann boundary conditions for $G$-symmetric functions is higher than $\bar{\l}_2$.

\medskip We now prove b). Let us point out that the quadratic form $Q$, as defined in \eqref{Q}, is perfectly well defined in $E$ even for $\rho=\rho_0$. And recall that $\rho_0$ is a point of degeneracy of $L$, that is, $Q^D$ is semipositive definite in the space $E_0$ (defined in \eqref{defE0}) and $Q^D(\psi)=0$ for some $\psi \in E_0$, $\psi \neq 0$. 

\medskip {\bf Claim: } We claim that $Q$ becomes negative in $E$ for $\rho= \rho_0$. Otherwise, $\psi$ would also be a minimizer for $Q$, and hence a critical point restricted to $E$. That is, for any $g \in E$, 
\begin{equation} \label{weak} \int_{B} \Big ( \rho_0 \nabla \psi \cdot \nabla g - \psi \phi + 3 (u_{\rho_0}^+)^2 \psi g \Big ) =0.
\end{equation}

Let us point out that the boundary term of the form $Q$ vanishes since $\psi=0$ on $\Sp$. We now claim that \eqref{weak} holds also if $g \in H^1_r(B)$ (this is the same argument as in the proof of Lemma \ref{tau}). By density, it suffices to consider radially symmetric functions $g \in C^2(\overline{B})$. By the definition of the space $E$ in \ref{defE}, we have that:

$$ \int_{B} \Big ( - \psi g + 3 (u_{\rho_0}^+)^2 \psi g  \Big )=0.$$

Moreover, if $g$ is radial, also $\Delta g$ is radially symmetric, and then:

\begin{eqnarray*} \int_{B} \rho_0 \nabla \psi \cdot \nabla g  = \rho_0 \left [\int_B - \psi \Delta g  + \int_{\partial B}  \psi \partial_{\nu} g    \right ] = 0. 
\end{eqnarray*}

Then, $\psi$ satisfies \eqref{weak} for any $\phi \in H^1(B)$, which implies that $\psi$ is a weak solution of the problem:

$$\begin{cases}
	- \rho_0 \Delta \psi - \psi + 3 (u_\rho^+)^2 \psi=0, & \mbox{ in }B, \\
	 \partial_{\nu} \psi =0 &\mbox{on }\partial B. \end{cases}$$

A standard bootstrap argument allows us to obtain $C^{2, \alpha}$ regularity of $\psi$. Recall now that $\psi=0$ on $\partial B$: by the unique continuation argument we conclude that $\psi=0$, a contradiction. The claim is proved.

\medskip 

By the claim, there exists $\psi \in E$ with $Q(\psi)<0$ for $\rho = \rho_0$. By a continuity argument, the same inequality holds for $\rho > \rho_0$ sufficiently close to $\rho_0$, and this proves b).

\end{proof}

The previous result allows us to define
\begin{equation*}
	\rho_{\ast}:=\sup\left\{\rho\in (\rho_{0},\bar{\lambda}_{2}^{-1}):Q_\rho(\psi)< 0 ~\mbox
	{for some}~\psi\in E\right\} \in (\rho_{0},\bar{\lambda}_{2}^{-1}).
\end{equation*}
From Proposition \ref{linear2}, Lemma \ref{tau} and the definition of $\rho_\ast$ we have the following result, which is the key to prove Theorem \ref{main}.
\begin{proposition} \label{keybif}
	The following assertions hold true:
	\begin{itemize}
		\item[(i)] There exists $\e>0$ such that for $\rho \in (\bar{\lambda}_2^{-1}-\e,\bar{\lambda}_2^{-1})$, $\tau_1(H_\rho)>0$.
		\item[(ii)] If $\rho \geq \rho_{\ast},$ then $\tau_{1}(H_\rho) \geq 0$.
		\item[(iii)] If $\rho = \rho_{\ast},$ then $\tau_{1}(H_\rho) = 0$.
		\item[(iv)] For any $\e>0$ there exists $\rho \in (\rho_{\ast}- \e, \rho_{\ast}),$ with $\tau_{1}(H_\rho)<0.$
	\end{itemize}
\end{proposition}

\section{Proof of Theorem \ref{main}}

In this section we use the previous analysis to conclude the proof of Theorem \ref{main}. In particular, 
Proposition \ref{keybif} is the key to prove the local bifurcation result, thanks to the Krasnoselskii bifurcation theorem. We state it below in a version which is suited for our purposes: for a proof, we refer to ~\cite{K04,S94} (see also \cite[Remark 6.3]{RRS20}).

\begin{theorem}[Krasnoselskii Bifurcation Theorem] \label{Kras} 
	Let $\mathcal{Y}$ be a Banach space, and let $\mathcal{W}\subset \mathcal{Y}$ and $\Gamma\subset\mathbb{R}$ be open sets, where we assume $0\in\mathcal{W}$. Denote the elements of $\mathcal{W}$ by $w$ and the elements of $\Gamma$ by $\rho$. Let $T:\Gamma \times \mathcal{W}\rightarrow \mathcal{Y}$ be a $C^{1}$ operator such that
	\begin{itemize}
		\item[i)] $T(\rho,0)=0$ for all $\rho\in\Gamma;$
		\item[ii)] $T(\rho,w)=w-K(\rho, w)$, where $K(\rho, w)$ is a compact map;
		\item[iii)] Let us denote by $i(\rho) $ the sum of the multiplicities of all negative eigenvalues of $D_{w}T(\rho,0)$. Assume that there exist $\rho_1, \ \rho_2 \in \Gamma$, ${\rho_1} <{\rho_2}$ such that:
		\begin{enumerate}
			\item $D_{w}T({\rho_i},0)$ are non degenerate, $i=1, \ 2$.
			\item $i({\rho_1})$ and $i({\rho_2})$ have different parity.
		\end{enumerate}
		
	\end{itemize}
	Then there exist $\tilde{\rho} \in ({\rho}_1, {\rho}_2)$ and a sequence $(\rho_n, w_n) \in \mathcal{Y}\times\mathbb{R}$, $w_n \neq 0$, such that $(\rho_n, w_n) \to (\tilde{\rho},0)$ and $T(\rho_n, w_n)=0$. 
\end{theorem}

Let us take $\e>0$, $\rho_1 \in (\rho_*-\e, \rho_*)$ so that $\tau_1(H_\rho)<0$ for $\rho= \rho_1$, and $\rho_2$ close to $\bar{\lambda}_2^{-1}$ such that $\tau_1(H_\rho)>0$ for $\rho= \rho_2$. This choice is possible by Proposition \ref{keybif}. In order to apply Theorem \ref{Kras}, we first need to reformulate the problem. First, since $\rho \in [\rho_1, \rho_2]$ the constant $c_\rho$ defined in \eqref{constant} is uniformly bounded from above, say, by $C>0$. By taking $\e>0$ sufficiently small, we can assume that $\tau_1(H_\rho) >- \frac{1}{2C}$ for all $\rho \in [\rho_1, \rho_2]$.

Take $\mathcal{V}$ is a neighborhood of $0$ in $ C^{2,\alpha}_{G,m}(\Sp) $ such that $\mathcal{U} \supset [\rho_1, \rho_2] \times \mathcal{V}$, where $\mathcal{U}$ is the neighborhood given in the definition of $F$ in \eqref{defF}. Let us define:
$$S:[\rho_1, \rho_2] \times \mathcal{V} \to C^{1,\alpha}_{G,m}(\Sp), \ \ \ S(\rho, w)= w -F(\rho, w).$$
Observe that 
$$D_{w} S(\rho, 0) = c_\rho \, H_{\rho}(w) +  w. $$ Then, the first eigenvalue of $D_{w} S(\rho, 0)$ is strictly positive, so it is invertible with respect to $w$. By the Inverse Function Theorem, we can take a neighborhood $\mathcal{W}$ of $0$ in  $C^{1,\alpha}_{G,m}(\Sp)$ so that $S$ is invertible in $\mathcal{W}$ for any $\rho \in [\rho_1, \rho_2]$.
We now define:
$$ T:[\rho_1, \rho_2] \times \mathcal{W} \to C^{1,\alpha}_{G,m}(\Sp), \ \ \ T(\rho, w)= w - S_\rho^{-1}(w).$$

Since the image of $S_\rho^{-1}$ is in $C^{2,\alpha}_{G,m}(\Sp)$, by the Ascoli-Arzela Theorem the operator $T$ has the form of identity minus a compact operator. Clearly, $T(w)=0$ if and only if $F(w)=0$. Moreover,

$$D_w T(\rho, 0) = \mu w \Leftrightarrow (1- \mu) D_wS(\rho, 0)= w \Leftrightarrow H_\rho(w)= \frac{\mu}{(1-\mu) c_\rho} w.$$

Therefore $D_wT(\rho, 0)$ has the same number of negative eigenvalues (with the same multiplicity) as $H$ (recall that by the choice of $\rho_1$, $\rho_2$, the eigenvalues of $H$ are bigger than $-1$).

\bigskip

Observe that the operator $T$ and the values $\rho_1$, $\rho_2$ fit in the setting of Theorem \ref{Kras}. In order to conclude the proof of Theorem \ref{main}, we just need to show that for any $\rho \in (\rho_0, \bar{\lambda}_2^{-1})$ such that $\tau_1(H_\rho)=0$, its multiplicity is odd. This is a consequence of the assumption \ref{G} on the symmetry group $G$, as we will see in the next lemma.

\begin{lemma} \label{odd}
Take $\rho \in (\rho_0, \bar{\lambda}_2^{-1})$ such that $\tau_1(H_\rho)=0$. Then, $Ker (H_\rho)$ has odd multiplicity. 
\end{lemma}
\begin{proof}
Let $w \in Ker(H_\rho)$, and $\psi_w \in E$ as given in Proposition \ref{psiv}. By Lemma \ref{E}, we can decompose $\psi_w$ in Fourier series as

$$ \psi_w = \sum_{k=1}^{+\infty} f_k(r) \vartheta_k(\theta).$$
Here $(r, \theta) \in (0,1] \times \Sp$ are spherical coordinates and $\vartheta_k$ denote $G$-symmetric spherical harmonics, normalized with respect to the $L^2$ norm. 

Then the quadratic form $\psi\mapsto Q_{\rho}(\psi)$ defined in $E$ can be given by
	\begin{equation*}
		Q_{\rho}(\psi)= \sum_{k=1}^{+\infty} Q_{\rho}^{k}(f_k),
	\end{equation*}	where the functional $Q_{\rho}^{k}$ is defined as
	\begin{align*}
		Q_\rho^k(f)&=\int_{0}^{1} \big (\rho f'(r)^{2}- f(r)^{2}+ 3 (u_{\rho}(r))^{2}f(r)^{2}) \big ) r^{N-1} \, dr   \\
		& +\rho \gamma_{k} \int_{0}^{1} f(r)^{2} r^{N-3}\, dr + \rho (N-1) f(1)^{2}.
	\end{align*}
Since $Q_\rho$ is semipositive definite, also $Q_\rho^k$ are semipositive definite. But $Q_\rho^k \geq Q_\rho^{1}$ for any $k$: this implies that $Q_\rho^k$ are positive definite for any $k \geq 2$. As a consequence, $\psi= f_1(r) \vartheta_1(\theta)$. Observe that $f_1$ is an eigenfunction of the problem in dimension 1, which have multiplicity 1. Moreover, $\gamma_1$ has odd multiplicity as an eigenvalue of the Laplace-Beltrami operator in $\Sp$, by assumption \ref{G} (see the Appendix in this regard). From this the proof of Lemma \ref{odd} follows.
\end{proof}

\section{Appendix}

This section is devoted to the study of the condition \ref{G} of the symmetry group $G$.  Let us consider the Helmholtz equation

\begin{equation*} -\Delta \phi = \phi \end{equation*}

Such equation admits solutions in the form $\phi(x)= f(r) \vartheta_k(\theta)$, where $\vartheta_k$ is a G-symmetric eigenfunction of the Laplace-Beltrami operator on $\Sp$. Observe that the corresponding eigenvalue $\gamma_k$ has the form$\gamma_k = i(i+N-2)$ for some $i=i(k)$ that depends on $G$. Moreover, $f$ solves the equation:

\begin{equation*} - f''(r) -(N-1) \frac{f'(r)}{r} + \gamma_k \frac{f(r)}{r^2}= f(r), \ \ f'(0)=0. \end{equation*}

The solutions of the above equation can be written in terms of the Bessel functions of the first kind $J_{\alpha}$ as:

\begin{equation} \label{bessel2} f(r)= r^{1-N/2} J_{\alpha}(r), \ \ \alpha= N/2+i(k)-1 \end{equation}

For radial functions, $k=i=0$ and $f(r)= r^{1-N/2} J_{N/2-1}(r)$. Then, the second eigenvalue of the Dirichlet problem in the unit ball $\bar{\lambda}_2$ is related to the second zero $r_2$ of $J_{N/2-1}(r)$, as
$$ \bar{\lambda}_2 = r_2^2.$$

This is an immediate consequence of the change of variables $\psi(x)= \phi(r_2x)$, $\psi:B \to \R$. \medskip

Moreover, the first nonradial eigenvalue $\sigma$ of the Laplacian under Neumann boundary conditions in $B_1$ is related to the first zero $s_1$ of $f'(r)$ where $f$ is given in \eqref{bessel2} with $i=i(1)$. The relation is, again,

\begin{equation*} \sigma = s_1^2. \end{equation*}
Then, to check assumption \ref{G} we need to show that
 \begin{equation} \label{a1} s_1 > r_2, \end{equation}
 and moreover that the space G-symmetric spherical harmonics related to $\gamma_1$ has odd dimension.
 
 \medskip 
 
At this point we need to estimate the roots of those expressions involving Bessel functions: we have used Mathematica for that. We will see that inequality \eqref{a1} imposes some lower bounds on $i(1)$.

\subsection{The case N=2}

If $N=2$, then one can estimate $r_2$  the second zero of $J_0(r)$ as:
$$ r_2 \sim 5.52008.$$
Now we are interested on the first zero of the derivative of $J_{i}(r)$: these are given in the following table:

\medskip

\begin{tabular}{| c | c |}
\hline    i $ $  &  First zero of $J_{i}'(r)$ \\
\hline	1 &    1.84118 \\
	2 &      3.05424 \\
3 &   4.20119 \\
4 &    5.31755 \\
5 &    6.41562 \\
\hline
\end{tabular}

\bigskip Then, it suffices to take a symmetry group $G$ such that $i(1) \geq 5$, or, in other words, a group $G$ excluding the first four spherical harmonics. Such spherical harmonics, in dimension 2, are just given by $\vartheta_j(\theta) = \cos(j \theta + \delta)$, $j=$ 1, 2, 3, 4. For instance, any dihedral symmetry group $\mathbb{D}_j$, $j \geq 5$, satisfies \ref{G} (and the multiplicity is equal to $1$).

\subsection{The case N=3}

If $N=3$, we have $J_{1/2}(r)= \sqrt{\frac{2}{\pi}} \,  \frac{\sin r}{\sqrt{r}}$ and its second positive root is clearly $$r_2= 2 \pi.$$ Now we present an estimate of the first zero of the derivative of $r^{-1/2} J_{1/2+i}(r)$ depending on $i$:

\bigskip

\begin{tabular}{| c | c |} 
 \hline  i $ $ &  First zero of 
	$\ \Big( r^{-1/2} J_{1/2+i}\Big)'(r)$  \\
	\hline 1 &    2.08158 \\
	2 &      3.34209 \\
	3 &   4.51410 \\
	4 &    5.64670 \\
	5 &    6.75646 \\
	\hline
\end{tabular}

\bigskip

Also here it suffices to take a symmetry group $G$ such that $i(1) \geq 5$: in other words, we require that the spherical harmonics associated to the first 4 eigenvalues are not $G$-symmetric. And we also need that $\gamma_1$ has odd multiplicity. One example of such group is the group of all isometries of the icosahedron, where $i(1)= 6$ and $\gamma_1=42$ is simple (see \cite{laporte}).

\subsection{The case N=4}

If $N=4$, the second zero of $J_{1}(r)$ is estimated as:
$$ r_2 \sim 7.01559.$$
We estimate the first zero of the derivative of $r^{-1} J_{1+i}(r)$ depending on $i$:

\bigskip

\begin{tabular}{| c | c |} 
	\hline  i $ $ &  First zero of 
	$\ \Big( r^{-1} J_{1 +i}\Big)'(r)$  \\
	\hline 1 &    2.29991 \\
	2 &      3.61126 \\
	3 &   4.81128 \\
	4 &    5.96235 \\
	5 &    7.08548 \\
	6 & 8.19039 \\
	\hline
\end{tabular}

\bigskip

Again, it suffices to take a symmetry group $G$ with $i(1) \geq 5$ such that $\gamma_1$ has odd multiplicity. An example of such group is the symmetry group of rotations of the hyper-icosahedron. The hyper-icosahedron is a regular polytope with 600 tetrahedral cells and 120 vertices, and its group of rotations forms a 7200 elements discrete subgroup of $SO(4)$ (see \cite{widom}). Because the efect of the inversion, all G-symmetric spherical harmonics of odd degree $i$ vanish. According to \cite{widom,  mamone}, for that group $G$ we have that $i(1)=12$ (the value $i$ corresponds to $n$ in the notation of \cite{widom}, and to $2l$ in the notation of \cite{mamone, widom}). 

The multiplicity is not written explicitly in \cite{widom}, but can be calculated from formulae (2.48) and (2.45) of that paper, taking into account the description of the rotations in $Y$ given in page 121. Then we can compute, for each even number $i$, the multiplicity of the space of G-symmetric spherical harmonics of order $i$, as:

\begin{align*} m(i)  & = \frac{1}{60} \left [20 \frac{\sin \Big (\frac{i+1}{2} \frac{2 \pi}{3} \Big)}{\sin \Big (\frac{1}{2} \frac{2 \pi}{3} \Big)}  + 12  \frac{\sin \Big (\frac{i+1}{2} \frac{2 \pi}{5} \Big)}{\sin \Big (\frac{1}{2} \frac{2 \pi}{5} \Big)}  + 12 \frac{\sin \Big (\frac{i+1}{2} \frac{4 \pi}{5} \Big)}{\sin \Big (\frac{1}{2} \frac{4 \pi}{5} \Big)}  \right. \\ &\left. + 15 \frac{\sin \Big (\frac{i+1}{2} \pi \Big)}{\sin \Big (\frac{1}{2} \pi \Big)} + i+1 \right ].\end{align*}

It turns out that $m(i)=0$ for $i=2$, $4$, $6$, $8$ and $10$ but $m(12)=1$, so that $G$ satisfies condition \ref{G}.

\bigskip

{\bf Data availability statement: } This manuscript has no associated data.

\end{document}